\documentclass[11pt]{article}
\usepackage[dvips]{graphicx}
\usepackage{amssymb,amsmath,color}
\usepackage{url}

\usepackage[mathcal]{eucal}

\DeclareMathAlphabet\mathbfcal{OMS}{cmsy}{b}{n}

\usepackage[colorlinks=true, allcolors=blue, pagebackref]{hyperref}       
\renewcommand*{\backrefalt}[4]{%
    \ifcase #1 \footnotesize{(not cited)}%
    \or        \footnotesize{(cited on page~#2)}%
    \else      \footnotesize{(cited on pages~#2)}%
    \fi}

\newcommand{\BEAS}{\begin{eqnarray*}}
\newcommand{\EEAS}{\end{eqnarray*}}
\newcommand{\BEA}{\begin{eqnarray}}
\newcommand{\EEA}{\end{eqnarray}}
\newcommand{\BEQ}{\begin{equation}}
\newcommand{\EEQ}{\end{equation}}
\newcommand{\BIT}{\begin{itemize}}
\newcommand{\EIT}{\end{itemize}}
\newcommand{\BNUM}{\begin{enumerate}}
\newcommand{\ENUM}{\end{enumerate}}
\newcommand{\BA}{\begin{array}}
\newcommand{\EA}{\end{array}}
\newcommand{\diag}{\mathop{\rm diag}}
\newcommand{\Diag}{\mathop{\rm Diag}}

\newcommand{\tr}{\mathop{ \rm tr}}

\newcommand{\idm}{I}
\newcommand{\rb}{\mathbb{R}}

\newcommand{\zb}{\mathbb{Z}}

\newcommand{\BlackBox}{\rule{1.5ex}{1.5ex}}  

\newenvironment{proof}{\par\noindent{\bf Proof\ }}{\hfill\BlackBox\\[2mm]}
\newtheorem{lemma}{Lemma}
\newtheorem{theorem}{Theorem}
\newtheorem{proposition}{Proposition}

\newcommand{\mysec}[1]{Section~\ref{sec:#1}}
\newcommand{\eq}[1]{Eq.~(\ref{eq:#1})}
\newcommand{\myfig}[1]{Figure~\ref{fig:#1}}

\def \E{{\mathbb E}}

\def \X{{\mathcal X}}
\def \V{{\mathcal V}}
\def \Y{{\mathcal Y}}
\def \K{{\mathcal K}}

 \newcommand{\bimat}[4]{\bigg(  \! \! \begin{array}{c@{\hspace{2mm}}c} {#1} &  {#2} \\ {#3} &  {#4} \end{array}   \! \! \bigg)}

\newcommand{\Hsquare}{
  \  \text{\fboxsep=-.2pt\fbox{\rule{0pt}{1ex}\rule{1ex}{0pt}}} \ 
}

\def \ds{\displaystyle}

\title{Sum-of-squares relaxations for polynomial \\
min-max problems over simple sets}

\author{Francis Bach\\
Inria,  Ecole Normale Sup\'erieure \\
PSL Research University \\
{\small \url{francis.bach@inria.fr}}}

\date{\today}

\begin{document}
\maketitle

   \begin{abstract}
   We consider min-max optimization problems for polynomial functions, where a multivariate polynomial is maximized with respect to a subset of variables, and the resulting maximal value is minimized with respect to the remaining variables. When the variables belong to simple sets (e.g., a hypercube, the Euclidean hypersphere, or a ball), we derive a sum-of-squares formulation based on a primal-dual approach. In the simplest setting, we provide a convergence proof when the degree of the relaxation tends to infinity and observe empirically that it can be finitely convergent in several situations. Moreover, our formulation leads to an interesting link with feasibility certificates for polynomial inequalities based on Putinar's Positivstellensatz.
    \end{abstract}
   
 \section{Introduction}
 In this paper, we consider min-max optimization problems of the form
\BEQ 
\label{eq:minmax} \min_{x \in \X}\   \max_{y \in \Y} \ g(x,y),
\EEQ
where $\X$ and $\Y$ are compact sets and $g$ is a continuous function. Throughout the paper, like \cite{lasserre2011min}, we will assume that $g$ is a multivariate polynomial. Among particular cases, a finite set $\Y$ leads to the minimization of the maximum of multivariate polynomials, which is typically not a polynomial function. Thus min-max problems extend the reach of polynomial optimization and have applications in several areas, such as robust optimization~\cite{ben2009robust}. Note that we do not consider saddle-point problems where polynomial optimization has already been studied~\cite{nie2021saddle}.

We will consider algorithms based on the sum-of-squares principle~\cite{lasserre2001global,parrilo2003semidefinite}. This problem has been looked at by \cite{lasserre2011min}, which models the function $x \mapsto \max_{y \in \Y} g(x,y)$ by a polynomial, as an upper-bound that is tightly converging as the degree of the approximant increases, but slowly and in most interesting cases non finitely. This bound is then minimized in a two-stage approach, which can deal with a set~$\Y$ which can be defined through polynomial inequalities. In this paper, we will need to assume that both sets $\X$ and~$\Y$ are ``simple'', in a sense to be defined in \mysec{sos}. This includes the regular hypercube, the Boolean hypercube, the unit Euclidean sphere or ball, and all Cartesian products of such sets. However, we will consider a one-stage primal-dual approach that is often finitely convergent (although we currently do not have any provable sufficient conditions).

\paragraph{Paper outline.}
We review SOS relaxations over simple sets in \mysec{sos} and present our SOS formulation for the min-max problem in \mysec{sosminmax}, together with algorithms based on kernels and a convergence proof, while in \mysec{experiments}, we perform illustrative experiments. We start by presenting in \mysec{duality} the duality principles that underlie our formulations, which apply beyond polynomials.

\section{Primal-dual formulations}
We first consider the classical primal-dual formulation of minimization problems, before extending it to min-max problems.
In this section, we consider continuous functions (not necessarily polynomials).

\label{sec:duality}

\subsection{Minimization problems}
\label{sec:dumin}
Given a continuous function $f$ defined on a compact set $\X$, minimizing $f$ can be cast as the minimization of 
$\int_\X f(x) d\mu(x)$ over $\mu \in \mathcal{P}(\X)$, the set of probability distributions on $\X$, that is,
\BEQ
\label{eq:sosconvdual}
\min_{x \in \X} f(x) = \min_{\mu \in \mathcal{P}(\X)} \ \int_\X f(x) d\mu(x),
\EEQ
where the minimizer is any measure supported on the minimizers of $f$.
Introducing the notation $\mathcal{M}(\X, Q)$ for the set of finite measures with values in the cone $Q$, we can see probability distributions as the elements of $\mathcal{M}(\X, \rb_+)$ such that $\int_\X d\mu(x) = 1$. Introducing a Lagrange multiplier $c \in \rb$ for this linear constraint, we get by convex duality:
\BEA
\notag \min_{\mu \in \mathcal{P}(\X)} \ \int_\X f(x) d\mu(x)
& = & \max_{c \in \rb} 
\inf_{\mu \in \mathcal{M}(\X, \rb_+)}\ \int_\X f(x) d\mu(x) + c \Big( \int_\X d\mu(x) - 1 \Big)
\\
\label{eq:sosconvprim}& = & \max_{c \in \rb} c\ \mbox{ such that } \ \forall x \in \X, \ f(x) -c \geqslant 0,
\EEA
which is equivalent to finding the largest minorant of $f$ (and thus provides a direct proof of strong duality). As shown in \mysec{sos}, these two equivalent formulations lead to equivalent SOS relaxations, by replacing non-negative functions by sums-of-squares in \eq{sosconvprim}, and representing probability measures by their moments and ``pseudo''-moments in \eq{sosconvdual}. We now extend these equivalent formulations to min-max problems.

\subsection{Min-max problems}
 \label{sec:duminmax}
We now consider primal-dual interpretations for the original problem in \eq{minmax}, akin to \eq{sosconvdual} and \eq{sosconvprim} in \mysec{dumin} above, for a continuous function $g: \X \times \Y \to \rb$.

For the outer minimization problem in $x \in \X$, we consider the probabilistic formulation from \eq{sosconvdual}, and we thus have the equivalent formulation:
\[
\min_{x \in \X}\  \max_{y \in \Y} \ g(x,y) = \min_{\mu \in \mathcal{P}(\X)} \int_\X  \Big( \max_{y \in \Y} \ g(x,y)  \Big) d\mu(x).
\]
For the inner maximization problem in $y \in \Y$, which is different for every $x \in \X$, we consider probability measures $\nu(\cdot|x) \in \mathcal{P}(\Y)$ (the set of probability measures on $\Y$), and apply the same reformulation, to obtain
\BEQ
\label{eq:F1}
\min_{x \in \X} \ \max_{y \in \Y} \ g(x,y) 
=
\min_{ \mu \in \mathcal{P}(\X)}\  \max_{ \nu: \X \to \mathcal{P}(\Y)} \ 
\int_\X \int_\Y g(x,y) d\nu(y|x)  d\mu(x).
\EEQ
This is now a convex-concave min-max problem in infinite dimensions (while the original one in \eq{minmax} is typically not), with a bilinear objective and two convex domains, for which min and max can be swapped~\cite{sion1958general,jahn2020introduction}.
  We can now use convex duality to obtain either a minimization problem or a maximization problem.

We have, from \eq{F1}, by adding the Lagrange multiplier $c \in \rb$ for the constraint  $\ds \int_\X d\mu(x) = 1$:
\BEA
\notag
\min_{x \in \X} \ \max_{y \in \Y} \ g(x,y) 
 &=  & \min_{ \mu \in \mathcal{P}(\X)} \max_{ \nu: \X \to \mathcal{P}(\Y), \ c \in \rb}
\int_\X \int_\Y g(x,y) d\nu(y|x)  d\mu(x) + c \Big( 1-\int_\X d\mu(x)  \Big)\\
\label{eq:33} & = & \!\!\!\!\max_{ \nu: \X \to \mathcal{P}(\Y), \ c \in \rb} c\ \mbox{ such that } \ \forall x\in \X,   \int_\Y g(x,y) d\nu(y|x)  \geqslant c,
\EEA
which is a \emph{maximization} problem.

Alternatively, by convex duality, this equal to,  introducing in \eq{F1} a Lagrange multiplier $ \lambda \in \mathcal{M}(\X,\rb)$ for the constraint that $\forall x \in \X$, $\int_\Y d\nu(y|x) = 1$~\cite{jahn2020introduction}:
 \[
  \min_{\mu \in \mathcal{P}(\X) ,  \ \lambda \in  \mathcal{M}(\X,\rb)  } \ \max_{ \nu: \X \to \mathcal{M}(\Y,\rb_+)}
  \int_\X \Big(  \int_\Y g(x,y) d\nu(y|x)    \Big) d\mu(x) + \int_\X \bigg( 1 - \int_\Y d\nu(y|x) \bigg) d\lambda(x)  ,
 \]
which is equal to
\BEQ
\label{eq:otd} \min_{ \mu \in \mathcal{P}(\X), \   \lambda \in  \mathcal{M}(\X,\rb)  } \int_\X  d\lambda(x) \ \mbox{ such that } \ \forall  y \in \Y, \  \lambda \geqslant g(\cdot,y) \mu,
 \EEQ
 which is another convex formulation as a \emph{minimization} problem.

Overall we get three formulations which are all equivalent to the original problem (in \mysec{sosminmax}, our SOS formulation will also have these three equivalent formulations):
\BIT
\item \textbf{Minimization}, corresponding to \eq{minSOSpb} and \eq{minmin} in \mysec{sosminmax}:
\BEQ
\label{eq:min} \min_{ \mu \in \mathcal{P}(\X), \   \lambda \in  \mathcal{M}(\X,\rb)  } \int_\X  d\lambda(x) \ \mbox{ such that } \ \forall  y \in \Y, \  \lambda \geqslant g(\cdot,y) \mu.
 \EEQ
   
\item \textbf{Maximization}, corresponding to \eq{maxSOSpb} and \eq{maxmax} in \mysec{sosminmax}:
\BEQ
\label{eq:max}
\max_{ \nu: \X \to \mathcal{P}(\Y), \ c \in \rb} c\ \mbox{ such that } \ \forall x\in \X,   \int_\Y g(x,y) d\nu(y|x)  \geqslant c.
\EEQ
  
\item \textbf{Saddle-point}, corresponding to \eq{fullSOS} and \eq{fullfull} in \mysec{sosminmax}:
\BEQ
\label{eq:minmaxfull} \min_{ \mu \in \mathcal{P}(\X)} \max_{ \nu: \X \to \mathcal{P}(\Y)}
\int_\X \int_\Y g(x,y) d\nu(y|x)  d\mu(x).
\EEQ
\EIT

\paragraph{Non-convex formulation.}
By writing $d\lambda(x) = a(x) d\mu(x)$  for a certain function $a: \X \to \rb$, which is only possible for a dense subset of $  \mathcal{M}(\X,\rb)$, we get  an equivalent reformulation
\BEA
\label{eq:ncf}
\min_{ \mu \in \mathcal{P}(\X), \ a: \X \to \rb} \int_\X a(x) d\mu(x) \ \mbox{ such that } \ \forall (x,y) \in \X \times \Y, \ a(x) \geqslant g(x,y) ,\EEA
which is a non-convex formulation because the objective is non-convex. An alternating minimization algorithm starting from a measure $\mu$ with full support leads to the global optimum after one minimization with respect to $a$ (leading to $a(x) = \max_{y \in \Y} g(x,y)$), and then one minimization with respect to $\mu$ (leading to the minimizer of this function $a$).
When using SOS formulations for these two operations,  we exactly obtain the formulation of \cite{lasserre2011min} (see \eq{sosex}  in \mysec{sosex}).

 \paragraph{Optimal solutions.}
 With $c_\ast$ being the optimal value of \eq{minmax}, the optimal measure $\mu \in \mathcal{P}(\X)$ is any measure supported on the minimizers of $  x \mapsto \max_{y \in \Y} g(x,y)$.
 The optimal $\nu: \X \to \mathcal{P}(\Y)$ is such that for all $x \in \X$, $  \int_\Y g(x,y) d\nu(y|x)  \geqslant c_\ast$
 with equality at any minimizer~$x_\ast \in \X$. Therefore, at all minimizers $x_\ast \in \X$, we need $\nu(\cdot|x_\ast)$ to put mass only at maximizers of $y \mapsto g(x_\ast,y)$, but  this is \emph{not} required at other positions.
The optimal $\lambda$ is equal to an optimal $\mu$ times $\max_{y \in \Y} g(x,y)$.

\section{SOS relaxations for polynomials over simple sets}
\label{sec:sos}
In this section, we review existing work on minimizing polynomial functions over simple sets, which we cast as minimizing a quadratic form $f(x) = \varphi(x)^\top F \varphi(x)$ for a feature map $\varphi: \X \to \rb^m$, where $\X$ is a compact set. While we use specific notations that will make further developments easier to describe, this section follows the classical SOS formulations (see~\cite{lasserre2010moments,henrion2020moment} for a thorough review).

We use the denomination ``simple set'' to refer to a set $\X$ coming with its feature map $\varphi:\X \to \rb^m$ \emph{with unit norm}, that is, $\| \varphi(x) \|^2 = 1$ for all $x \in \X$ (for the Euclidean norm), and, which can be represented (potentially after transformation) as a multivariate polynomial (this thus imposes that $\X$ is a subset of $\rb^d$ for a specific $d$).

We will always assume that the constant mapping and the identity mapping $x \mapsto x$ can be obtained as a linear function of $\varphi(x)$ (this will be useful in recovering maximizers in \mysec{practical}). Moreover, we will only need to access the positive-definite kernel function $k: \X \times \X \to \rb$ defined as $k(x,y) =  \varphi(x)^\top \varphi(y)$ (and not access to the vector $\varphi$). Our unit norm normalization on $\varphi$ translates to $k(x,x)=1$ for all $x \in \X$.

We assume that the dimension of the span  of all $\varphi(x)$, $x \in \X$ is $m$, while the dimension of the span $\V_\varphi$ of all  $\varphi(x) \varphi(x)^\top \in \rb^{m \times m}$, $x \in \X$, is $m' \in [m, m(m+1)/2]$. Finally, we assume we can generate (typically, randomly) $m'$ points $x_1,\dots,x_{m'}$, such that $\varphi(x_i) \varphi(x_i)^\top $, $i=1,\dots,m'$, is a basis of~$\V_\varphi$.

The optimization problem and our solution will be invariant by invertible linear transformations, and we can choose the feature map so that the kernel is as simple as possible (note, however, that in terms of conditioning of the associated numerical linear algebra, some kernels are better than others).

All of our examples will be (subsets of) Euclidean unit spheres or products of  Euclidean spheres.

\subsection{Examples}
\label{sec:examples}
We will consider the following sets, feature maps, and kernel functions. Since our relaxations are based on approximating non-negative polynomials as sums-of-squares, that is, positive semi-definite quadratic forms in $\varphi$, we describe these SOS polynomials for some instances.

\BIT
\item \textbf{Discrete data}: $\X = \{1,\dots,m\}$ with orthonormal features $\varphi: \X \to \rb^m$, defined as $\varphi(x)_i = 1_{x=i}$. The corresponding kernel is $k(x,x') = 1_{x=x'}$, with $m'=m$.

\item \textbf{Trigonometric polynomials on $[0,1]$}: $\X \in [0,1]$ with $\varphi(x)_\omega = \frac{1}{(2r+1)^{1/2}} e^{2i \pi \omega x}$ for $\omega \in \{-r,\dots,+r\}$.\footnote{This feature is complex-valued but equivalent real-valued formulations with cosines and sines could be used. Since we only use kernel formulations, we do not need to pursue them explicitly.} The kernel is $ k(x,x') = \frac{\sin[ (2r+1) \pi(x-x')]}{(2r+1) \sin \pi (x-x') }$. This can be equivalently represented in the unit Euclidean sphere in $\rb^2$ with the bijection $\theta \mapsto (\cos 2\pi \theta, \sin 2\pi \theta)$, where the corresponding feature map spans all bivariate polynomials of degree $r$, with a kernel that can be taken to be equal to $k(y,y') = \frac{1}{2^r} (1 + y^\top y' )^{r}$ for $y,y' \in \rb^2$ of unit norm (we could construct one with Chebyshev polynomials to get the exact equivalence with the kernel above). We then have $m=2r+1$ and $m'=4r+1$.

\item \textbf{Polynomials on $[-1,1]$}: this is simply the projection of the case above by considering $y \in \rb^2$ such that $y_1^2+y_2^2 = 1$, and only considering functions of $y_1$. As shown in~\cite{bach2022exponential}, a polynomial in $y_1$ which is equal to an SOS polynomial on $y_1,y_2$  can be written as the sum $u(y_1) + (1-y_1^2) v(y_1)$ where $u$ and $v$ are univariate SOS polynomials.

\item \textbf{Hypersphere:} $\X = \{ x \in \rb^{d+1}, \ \|x\|_2^2 = x^\top x = 1\}$, with all functions that are multivariate polynomials of degree $r$. This corresponds to  $  m =    { d+r  \choose r}   +  { d+r-1 \choose r-1} $ and $  m' =    { d+2r  \choose 2r}   +  { d+2r-1 \choose 2r-1} $. We can choose the kernels $  k(x,x') = \frac{1}{r+1}  \sum_{i=0}^r \big(  {x^\top y} \big)^i$
or $  k(x,x') = \frac{1}{2^r} (1 + x^\top x' )^{r}$. We could also use generalized Legendre polynomials~\cite{frye} to get better-conditioned kernel matrices.

\item \textbf{Euclidean ball:} $\X = \{ x \in \rb^{d}, \ x^\top x  \leqslant 1\}$ can be seen as the projection of the hypersphere above to the first $d$ dimensions. When obtaining an SOS polynomial on the hypersphere, this translates for the Euclidean ball  
to a sum $u(x) + (1-\| x\|_2^2) v(x)$ where $u$ and $v$ are SOS polynomials.

\item \textbf{Products of one-dimensional spheres $\subset \rb^2$ $\Leftrightarrow$ trigonometric polynomials on $[0,1]^d$ $\Leftrightarrow$ regular polynomials on $[-1,1]^d$}: this is the tensor product of the univariate cases above; the kernel is then
$  k(y,y') = \prod_{i=1}^d \frac{1}{2^r} (1 + y_i^\top y_i' )^{r}$ for the polynomial representations, or alternatively
 $  k(x,x') =  \prod_{i=1}^d  \frac{\sin[ (2r+1) \pi(x_i-x_i')]}{(2r+1) \sin \pi (x_i-x_i') }$ for trigonometric polynomials. This then corresponds to multivariate polynomials of 
 \emph{maximal}\footnote{For a monomial $X_1^{\alpha_1} \cdots X_d^{\alpha_d}$, its degree is $\alpha_1+\cdots+\alpha_d$ and its \emph{maximal} degree is $\max\{\alpha_1,\dots,\alpha_d\}.$} degree~$2r$.
 As shown in~\cite{bach2022exponential}, a trigonometric SOS polynomial transferred to regular polynomials on $[-1,1]^d$ leads to a representation of Schmudgen's type~\cite{schmudgen2017moment}.
 
\item \textbf{Boolean hypercube $\X = \{-1,1\}^d$}: it can be seen as a sub-case of the hypersphere in dimension $d-1$ 
and radius $\sqrt{d}$, where quadratic forms are polynomials of degree $2r$.
We then have $m = \sum_{i=0}^{r} { d \choose i}$.
 \EIT

\subsection{Relaxation}

The SOS relaxation is obtained by first representing the minimization of $f$ as the maximization of a minorant~$c$ of $f$, that is, such that $f(x)-c \geqslant 0$ for all $x\in \X$, that is, \eq{sosconvprim} in \mysec{duality}. We then represent non-negative functions as sums-of-squares, that is, a positive semi-definite quadratic form in $\varphi(x)$,  thus solving:
\BEQ
\label{eq:primalc}
 \max_{c \in \rb, \ A \succcurlyeq 0 } \ \ c   \ \  \mbox{ such that }\   \forall x \in \X, \ f(x) = c + \varphi(x)^\top A \varphi(x).
\EEQ
It can be re-written using $\mathcal{V}_\varphi$ the span of all $\varphi(x) \varphi(x)^\top$, $ x \in \X$, and its orthogonal subspace $\mathcal{V}_\varphi^\perp$, as:
\BEA
\notag & &  \max_{c \in \rb, \ A \succcurlyeq 0 } \ \ c  \ \ \mbox{ such that }\ \ \forall x \in \X, \ \tr \big[ \varphi(x) \varphi(x)^\top ( F - c \idm - A) \big] = 0 \\
\notag & = & \max_{c \in \rb, \ A \succcurlyeq 0, \ Y \in   \mathcal{V}_\varphi^\perp } \  c \  \ \mbox{ such that }\ \  F - c \idm - A + Y = 0 , \ \mbox{ by definition of } \mathcal{V}_\varphi^\perp.
\EEA
We can then optimize out $c$ and $A$, by noticing that $c \in\rb$ is the largest $c$ such that $F + Y \succcurlyeq c \idm$, leading to the following spectral formulation
\BEA
\label{eq:primalSOS} &   & \max_{Y \in  \mathcal{V}_\varphi^\perp }\   \lambda_{\min}(F + Y).
\EEA
Its dual can be written as, using standard semi-definite programming duality~\cite{Boyd2004Convex}:
\BEA
\notag \max_{Y \in  \mathcal{V}_\varphi^\perp }\   \lambda_{\min}(F + Y)
\notag  & = & \min_{ \Sigma \succcurlyeq 0 } \ 
\max_{Y \in  \mathcal{V}_\varphi^\perp }\   \tr [ \Sigma(F + Y)] \ \mbox{ such that } \ \tr(\Sigma) = 1 \\
\label{eq:dualSOS} & = & \min_{ \Sigma \succcurlyeq 0 } \ 
   \tr ( \Sigma F )\ \mbox{ such that } \ \tr(\Sigma) = 1 , \ \Sigma \in \mathcal{V}_\varphi,
\EEA
which corresponds to an outer approximation of the convex hull of all $\varphi(x)\varphi(x)^\top$, $x \in \X$, by the set of positive semi-definite matrices such that 
$\tr(\Sigma) = 1$ and $\Sigma \in \mathcal{V}_\varphi$, which we denote $\widehat{\K}_\varphi$ and  which is an outer approximation of $\K_\varphi$, the closure of the convex hull of all $\varphi(x)\varphi(x)^\top$, $x \in \X$. This dual formulation corresponds to (a) replacing the minimization of $f$ by the minimization with respect to a probability measure on $\X$ of the expectation of $f$ with respect to that measure, as done in \mysec{dumin} in \eq{sosconvdual}, and (b) characterizing these measures by their expectations of $\varphi \varphi^\top$. Elements of $\K_\varphi$ are moment matrices while elements of $\widehat{\K}_\varphi$ are often referred to as ``pseudo-moment'' matrices.
  
   \subsection{Kernelization}
   \label{sec:kernelization}
  With an explicit description of $\V$, it may be cumbersome to implement the semi-definite program, particularly for larger input dimensions, leading to dedicated codes for each case. This is simpler with kernels, as described below. It allows accessing the function $f$ using only function values, like proposed by~\cite{lofberg2004coefficients}, with a direct link with positive definite kernels outlined by~\cite{rudi2020finding}. As we now show, this corresponds to representing the space $\V$ by a span of finitely many elements, leading to a representation of the moment matrices $\Sigma$ as a linear combination of rank-one matrices.

   We consider $m'$ ``well-positioned'' points $x_1,\dots,x_{m'} \in \X$, so that $\V$ is the span of all $\varphi(x_i)\varphi(x_i)^\top$, $i=1,\dots,m'$. Quasi-random sequences~\cite{morokoff1994quasi} are natural candidates, in particular, because we will extract below the first $m$ points and also need them to be well-spread to avoid ill-conditioning of the kernel matrices.
   
   For the primal formulation in \eq{primalc}, the constraint that  $\forall x \in \X, \ f(x) = c + \varphi(x)^\top A \varphi(x) $ is equivalently replaced by an equality only on $x_1,\dots,x_{m'}$. This corresponds to checking that two polynomials are equal by checking that they are equal on sufficiently many points.
   
 The dual formulation in 
 \eq{dualSOS} is then equivalent to:
   \[
   \inf_{ \alpha \in \rb^{m'}} \ \sum_{i=1}^{m'} \alpha_i f(x_i) \mbox{ such that } \sum_{i=1}^{m'} \alpha_i = 1 , \ \sum_{i=1}^{m'} \alpha_i \varphi(x_i) \varphi(x_i)^\top \succcurlyeq 0,
   \]
   which is only accessing the function $f$ through $m'$ function evaluations. The crucial point is that the vector $\alpha \in \rb^{m'}$ is not constrained to have non-negative values (otherwise, the formulation above would lead to $\min_{i \in \{1,\dots,m'\}} f(x_i)$).
   
   If $m$ is the dimension of $\varphi$, then from the kernel matrix $K \in \rb^{m \times m}$ associated with the first $m$ points, we build the ``empirical feature map'' as 
   $
   \tilde{\varphi}(x) = K ^{-1/2} ( k(x_i,x))_{i \in \{1,\dots,m\}}  \in \rb^m $,
   where $K ^{-1/2}$ is any inverse square root of $K \in \rb^{m \times m}$. 
   This defines an empirical feature matrix
   $\Phi  = L K^{-1/2} \in \rb^{m' \times m}$, where $L \in \rb^{m' \times m'}$ is the full kernel matrix of all $m'$ points. We then solve, equivalently,
    \BEQ
\label{eq:sdpsos}   \inf_{ \alpha \in \rb^{m'}} \ \sum_{i=1}^{m'} \alpha_i f(x_i) \mbox{ such that } \sum_{i=1}^{m'} \alpha_i = 1 , \ 
   \Phi^\top \diag(\alpha) \Phi\succcurlyeq 0,
   \EEQ
   and obtain a solution $  \Sigma = \sum_{i=1}^{m'} \alpha_i \varphi(x_i) \varphi(x_i)^\top$. We will see below how to obtain a candidate maximizer $x_\ast \in \X$ from $\Sigma$ without the need to compute $\varphi$.

  \paragraph{Going infinite-dimensional.} Solving \eq{sdpsos} will lead to the SOS relaxation if $f$ is indeed a quadratic form in $\varphi(x)$. In all our examples, the feature map is finite-dimensional. Still, we can go infinite-dimensional using positive definite kernels corresponding to infinite dimensional feature spaces, such as $k(x,x') = \exp( x^\top x')$, with an additional regularizer. See~\cite{rudi2020finding} for more details and convergence analysis.
   
   \subsection{Practical algorithms}
\label{sec:practical}
\paragraph{Solving the SDP.}   The problem in \eq{sdpsos} is a semi-definite program, which can either be solved using generic toolboxes, with complexity $(m')^{3.5}$~\cite{helmberg1996interior}. Adding a log-determinant barrier leads to an approximate algorithm with only matrix inversions of size $m$ and $m'$~\cite{rudi2020finding}, but no eigenvalue decompositions.

\paragraph{Obtaining rank-one solutions.}
   The obtained solution $\alpha \in \rb^{m'}$ of \eq{sdpsos} may not lead to a rank-one matrix $  \Sigma = \sum_{i=1}^{m'} \alpha_i \varphi(x_i) \varphi(x_i)^\top$
   when the minimization problem has several minimizers or the relaxation is not tight. We can obtain a lower-rank solution (and rank-one when the relaxation is tight) by minimizing a random linear function of~$\alpha$ over all~$\alpha$ that are minimizers of \eq{sdpsos}. Rank-minimization heuristics could also be used~\cite{fazel2001rank}.

\paragraph{Obtaining candidates for $x_\ast$.} Once the vector $\alpha$ is obtained such that the matrix $\Sigma = \sum_{i=1}^{m'} \alpha_i \varphi(x_i) \varphi(x_i)^\top$ has rank one, we can simply obtain the corresponding $x_\ast \in \X$ exactly as $x_\ast = \sum_{i=1}^{m'} \alpha_i x_i$. This is only approximate when $\Sigma$ does not have rank one.
See also~\cite{henrion2005detecting}.
  
  \subsection{Tightness guarantees}
\label{sec:sosguarantees}
For a small number of cases, we have $\widehat{\K}_\varphi = \K_\varphi$, that is, the relaxation is tight, e.g., for one-dimensional problems or with linear features (modeling quadratic polynomials). Otherwise, we need ``hierarchies''.

\paragraph{Hierarchies.}
For most cases, the relaxation is not tight, that is, $\widehat{\K}_\varphi \supsetneq \K_\varphi$, but we can see a $2r$-dimensional polynomial as an instance of a polynomial of degree less than $2s$, for $s>r$, and run the algorithm with the kernel corresponding to this larger dimensional space (which requires access to more function values since it leads to an increase in $m'$). This corresponds to using a relaxation $\widetilde{\K}_\varphi$ such that $ {\K}_\varphi  \subset  \widetilde{\K}_\varphi  \subset  \widehat{\K}_\varphi$, for which $\sup_{\Sigma \in \widetilde{\K}_\varphi} \inf_{ \Sigma' \in {\K}_\varphi} \| \Sigma - \Sigma' \|_{\rm F}$ is hopefully going to zero when the degree~$s$ goes to infinity, where $\| \cdot \|_{\rm F}$ denotes the Frobenius norm. This is the case for several of the simple sets in the examples above, with a rate in $O(1/s^2)$, for hyperspheres~\cite{fang2021sum}, polynomials on $[-1,1]^d$~\cite{laurent2022effective}, and trigonometric polynomials~\cite{bach2022exponential}. By increasing the degrees $s$ until approximating the global optimum arbitrarily well, we obtain a ``hierarchy'' of optimization problems.

More precisely, this corresponds to replacing $\varphi(x) \in \rb^m$ by $\tilde{\varphi}(x) = { \varphi(x) \choose \varphi^+(x) } \in \rb^{ \tilde{m}}$, and $F$ by $\tilde{F} = \bimat{F}{0}{0}{0}$, with the function $f$ defined by $f(x) = \varphi(x)^\top F \varphi(x) = \tilde{\varphi}(x)^\top \tilde{F} \tilde{\varphi}(x)$. The convergence results in~\cite{fang2021sum,laurent2022effective,bach2022exponential} correspond to the existence of $\varepsilon(\varphi,\varphi^+) >0$ such that
\[
\forall x \in \X, \ f(x) \geqslant \varepsilon(\varphi,\varphi^+) \| F\|_{\rm F} \ \ \Rightarrow \ \ \exists \tilde{A} \succcurlyeq 0, \ \forall x \in \X, \ f(x) = \tilde{\varphi}(x)^\top \tilde{A} \tilde{\varphi}(x).
\]
The constant $\varepsilon(\varphi,\varphi^+)$ can be chosen as $O(1/s^2)$, where $s$ is the degree of the polynomials defining $\tilde{\varphi}$.

Note that in practice, when using kernel formulations, using hierarchies simply means using a different kernel and more function evaluations.

\subsection{Matrix-valued SOS}
\label{sec:matrixSOS}
In \mysec{sosminmax}, we will need to consider functions from $\X$ to some subspace $\mathcal{T}$ of $\mathbb{S}_p$ (the set of symmetric matrices of dimension $p$), and use characterizations of functions $f: \X \to \mathcal{T}$ that are linear in $\varphi(x) \varphi(x)^\top$ and such that for all $x \in \X$, $f(x) \succcurlyeq 0$. We assume the identiy matrix $\idm$ belongs to $\mathcal{T}$.

This is an extension of the classical situation (where $p=1$). We denote by $F \in \rb^{mp \times mp}$ the linear form defined with blocks $F_{ij}$ of size $m \times m$, for $i,j \in \{1,\dots,p\}$, such that 
  \[
  f(x) = F[ \varphi(x) \varphi(x)^\top],
  \]
  which is defined as $\forall x \in \X, \ f(x)_{ij} = \varphi(x)^\top F_{ij} \varphi(x)$. The constraint that for all $x \in \X$, $f(x) \in \mathcal{T}$ is equivalent to $F \in  \V_\varphi^\perp \otimes \mathbb{S}_p + \mathbb{S}_m \otimes \mathcal{T}$. Following~\cite{scherer2006matrix,fang2021sum,muzellec2021learning}, a sufficient condition for the matrix-non-negativity of $f$ is $F \succcurlyeq 0$. 

The condition $F \succcurlyeq 0$ is also necessary for some special cases. Indeed, if $\mathcal{T}$ is the set of diagonal matrices, we are then simply looking at $p$ different non-negative polynomials, and if $\varphi$ is such that we have a tight scalar SOS representation of non-negative functions, the condition is indeed necessary.

For the cases where we had the tightness guarantees in \mysec{sosguarantees}, it turns out that we have similar tightness guarantees, that is, if $f$ is a degree $2r$ matrix-valued polynomials. We consider $\tilde{\varphi} = { \varphi \choose \varphi^+  } $ leading to polynomials of degree $2s$, then if $f$ has strictly positive-semidefinite values (that is, all eigenvalues greater than $\varepsilon$ times some norm of $f$), then $f$ is a matrix-SOS polynomial of degree $2s$. The constant~$\varepsilon$ can be taken as $O(1/s^2)$.

Indeed, all the proofs for hyperspheres~\cite{fang2021sum}, polynomials on $[-1,1]^d$~\cite{laurent2022effective}, and trigonometric polynomials~\cite{bach2022exponential} are based on the same integral operator idea from~\cite{fang2021sum} who showed how to extend it to the matrix domain. See a precise instance of such a result for trigonometric polynomials in Appendix~\ref{app:matrixSOS}.

\paragraph{Link between matrix-SOS to tensor products.} In the min-max problem in \mysec{sosminmax}, we will need to minimize quadratic forms in $\varphi(x) \varphi(x)^\top \in \rb^{m \times m}$, rather than in $\varphi(x) \in \rb^m$. Such a quadratic form is defined as $f(x) = \tr\big[ F ( \varphi(x) \varphi(x)^\top  \otimes \varphi(x) \varphi(x)^\top) \big]$, and if the matrix-valued function $G: x \mapsto F[ \varphi(x) \varphi(x)^\top]$ is such that $\forall x \in \X, \ G(x) \succcurlyeq c \idm$, then  
 $f(x)  = \varphi(x)^\top G(x) \varphi(x) \geqslant c$ for all $x \in \X$. Thus the threshold for scalar-valued quadratic forms in $\varphi(x)\varphi(x)^\top$ leads to (at least) the same threshold for matrix-valued quadratic forms in~$\varphi(x)$.

\section{SOS relaxations for min-max problems}
\label{sec:sosminmax}
  
We consider the min-max problem 
\BEQ
\label{eq:SP} \min_{x \in \X} \ \max_{y \in \Y} \ g(x,y),
\EEQ
for a continuous function $g: \X \times \Y \to \rb$ defined on the product of two compact sets $\X$ and $\Y$.
We assume that we have two feature maps $\varphi:\X \to \rb^m$  and $\psi: \Y \to \rb^p$, such that $\| \varphi(x)\|=\| \psi(y)\|=1$ for all $x \in \X$ and $y \in \Y$, thus within the framework of \mysec{sos}. While the motivation is polynomials, this is not needed in most of this section.

We assume that the function $g$ is a bilinear function of $\varphi \varphi^\top$ and $\psi \psi^\top$, that is,  of the form
\BEQ
\label{eq:A}
g(x,y) = \tr \Big[ G  \big( \psi(y)\psi(y)^\top \otimes \varphi(x)\varphi(x)^\top  \big) \Big],
\EEQ
for a symmetrix matrix $G \in \rb^{mp \times mp}$. 
By definition of the Kronecker product~\cite{golub83matrix}, we can see $G$ as matrix defined by blocks $G_{ij}$ of size $m \times m$, for $i,j \in \{1,\dots,p\}$, and \eq{A} can be rewritten as:
\[
g(x,y) = \sum_{i,j=1}^p \psi(y)_i \psi(y)_j \cdot \varphi(x)^\top G_{ij} \varphi(x).
\]
Because of our unit norm assumptions for the feature maps, this includes linear forms in $\varphi \varphi^\top$ and $\psi \psi^\top$ (e.g., by considering all $G_{ij}$ proportional to $\idm$, we obtain a linear form in $\psi \psi^\top$). For the examples in \mysec{sos}, such a representation exists for all multivariate polynomials in $x$ and $y$.

The goal of this paper is to design SOS methods for this problem. Note that they will sometimes not be relaxations per se, as their values will not always be lower bounds on optimal values.

\paragraph{Notations.}
Following~\mysec{sos}, we denote by $\V_\varphi \subset \rb^{m \times m}$ the span of all $\varphi(x) \varphi(x)^\top$, and $\K_\varphi  \subset \rb^{m \times m} $ the closure of its convex hull, with similar notations for $\V_\psi  \subset \rb^{p \times p} $ and $\K_\psi  \subset \rb^{p \times p} $.

The natural SOS formulation is to replace $\K_\varphi$ by $\widehat{\K}_\varphi = \{ S \in \V_\varphi, \ S \succcurlyeq 0, \ \tr(S) = 1\} \supset \K_\varphi$ and
$\K_\psi$ by $\widehat{\K}_\psi = \{ T \in \V_\psi, \ T \succcurlyeq 0, \ \tr(T) = 1\} \supset \K_\psi$.
When using hierarchies, we may use tighter sets $\widetilde{\K}_\varphi$ and $\widetilde{\K}_\psi$, which often corresponds to embedding $\varphi$ in a bigger feature map.

We will also need $\mathcal{K}_{\varphi \otimes \varphi} \in \rb^{m^2 \times m^2}$ corresponding to the hull of all $\varphi(x)^{\otimes 4} = 
\varphi(x)\varphi(x)^\top \otimes \varphi(x)\varphi(x)^\top$, $x \in \X$, as well as it outer approximation 
$\widehat{\mathcal{K}}_{\varphi \otimes \varphi}$.

 \subsection{Existing SOS relaxation} 
 \label{sec:sosex}
 The method of~\cite{lasserre2011min}, which applies more generally (in particular to sets $\Y$ which are not simple), corresponds to an SOS formulation for \eq{ncf}, for a fixed probability measure $\mu \in \mathcal{P}(\X)$ with full support, that is,
\[
 \min_{   a: \X \to \rb} \int_\X a(x) d\mu(x) \ \mbox{ such that } \ \forall (x,y) \in \X \times \Y, \ a(x) \geqslant g(x,y)
 ,
 \]
 and then the minimization of the resulting function $a$.
 
 It can be cast as follows with our notations. In a first stage, assuming that one can compute $\Sigma = \E [ \varphi(x) \varphi(x)^\top]$ for a distribution with full support on $\X$  (typically the uniform distribution), we solve
 \BEA
\label{eq:jb} 
\label{eq:sosex} & & \min_{ A \in \rb^{m \times m}, \ C \in \rb^{mp \times mp} }\tr (A \Sigma) \\
 \notag & & \mbox{ such that }
 C \succcurlyeq 0 \mbox{ and } \forall (x,y) \in \X \times \Y, \ g(x,y) = \varphi(x)^\top A \varphi(x) -  \tr \big[ C  \big( \psi(y)\psi(y)^\top  \otimes \varphi(x)\varphi(x)^\top  \big) \big],
 \EEA
 which approximates, with an SOS approach, the polynomial in $x$ with the smallest expectation, which is above $g(x,y)$ for all $(x,y) \in \X \times \Y$. In the second stage, this polynomial  $a$ defined by the matrix $A$ is minimized with an SOS method. This will converge when the degree of $a$ is allowed to increase but requires approximating a non-polynomial function by a polynomial function, which may require a large degree. Moreover, it is typically not finitely convergent.
 Note that we could also minimize with respect to~$\Sigma$ in \eq{jb}, but this leads to a non-convex problem. A natural algorithm for this non-convex problem is to perform alternating optimization, alternating between optimizing with respect to $A$ and $\Sigma$, which improves the result but is not globally convergent in general (see Appendix~\ref{app:alternate} for more details). Finally, if the polynomial defined by $A$ is minimized exactly, we obtain an upper-bound on the actual optimal value of \eq{SP}.

\paragraph{Kernelization.} Using notations from \mysec{sos},
 we can formulate the problem in \eq{jb} above as
  \BEAS
&& \min_{ A \in \rb^{m \times m}, \ C \in \rb^{mp \times mp} }\tr (A \Sigma) \mbox{ such that }
 C \succcurlyeq 0 \mbox{ and } G   - \idm \otimes A    +  C   \in (\V_\varphi \otimes \V_\psi)^\perp
 \EEAS
 by definition of the vector space $\V_\varphi \otimes \V_\psi$. We can then introduce a Lagrange multiplier $M \in \V_\varphi \otimes \V_\psi$ to obtain by convex duality:
 \BEAS
 &  & 
 \max_{M \in \V_\varphi \otimes \V_\psi} \ 
 \min_{ A \in \rb^{m \times m}, \ C \succcurlyeq 0 } \ \tr (A \Sigma)  + \tr \big[ M ( G   - \idm \otimes A     +  C  ) \big].
\EEAS
We can then optimize with respect to $C \succcurlyeq 0$, which leads to the constraint $M \succcurlyeq 0$, and with respect to $A$, which leads to a linear constraint, that is:
\BEAS
 &  & 
 \max_{M \in \V_\varphi \otimes \V_\psi} \tr ( M  G) \mbox{ such that } M \succcurlyeq 0 \mbox{ and } \widetilde{\tr}[M] =   \Sigma,
 \EEAS
 where  $\widetilde{\tr} [M] \in \rb^{m \times m} $ denotes the ``partial trace'' defined as $\tr( N \widetilde{\tr} [M]  ) = \tr( M(N \otimes \idm))$ for any matrix $N \in \rb^{ m \times m}$, and $\mathbb{S}_p$ denotes the set of symmetric matrices of size $p$.  In other words, $(\widetilde{\tr} [M])_{ij} = \tr( M_{ij})$.

 It can be kernelized like in \mysec{kernelization}, in particular in situations where $\Sigma = \frac{1}{m} \idm$, which is the case when using a uniform distribution and $\varphi$ obtained from orthonormal bases. Like in \eq{sdpsos}, we can then represent~$\Sigma$ as $\Sigma = \sum_{i=1}^{m'} \mu_i \varphi(x_i) \varphi(x_i)^\top$ for some $\mu \in \rb^{m'}$. We thus solve
 \BEQ
\label{eq:sosexk} \max_{ \alpha \in \rb^{m'  \times p'}} \sum_{i,j} \alpha_{ij} g(x_i,y_j) \mbox{ such that } \forall i, \sum_{j} \alpha_{ij} = \mu_i
 \mbox{ and } \sum_{i,j} \alpha_{ij} \varphi(x_i)\varphi(x_i)^\top \otimes \psi(y_j) \psi(y_j)^\top \succcurlyeq 0,
 \EEQ
 and we recover $\varphi(x_i)^\top A \varphi(x_i)$ from the Lagrange multiplier for the constraint $\sum_{j=1}^{p'} \alpha_{ij} = \mu_i$. This is sufficient to minimize $a(x) = \varphi(x)^\top A \varphi(x)$ with an SOS method like described in \mysec{sos}.
 
 \paragraph{Illustration.} We consider $\X = [0,1]$ and trigonometric polynomials, with $\Y = \{1,2,3\}$. Thus, we aim to minimize the maximum of three trigonometric polynomials, which we take to have a maximal degree of $2$. This is illustrated in \myfig{lasserre}, where we plot the three polynomials and the upper-bounding polynomial when using polynomials of degree $r$, with $r=2, 4, 8$, where we can see the slow and in general only asymptotic convergence.
 
 \begin{figure}
\begin{center}
 \includegraphics[width=14cm]{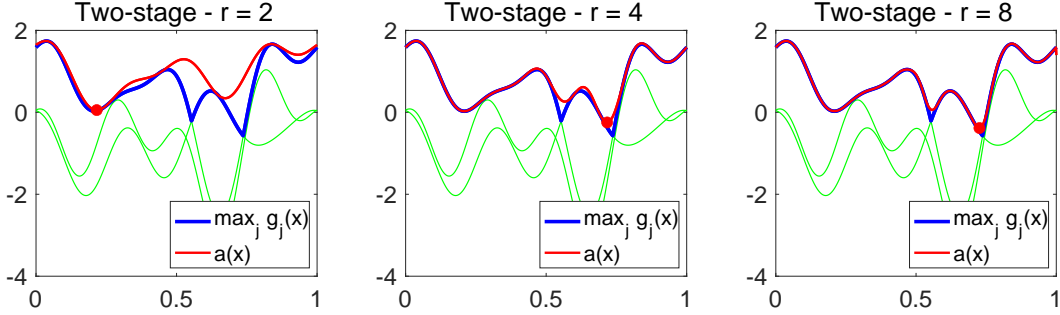}
\end{center}
\vspace*{-.45cm}

 \caption{Two-stage approach~\cite{lasserre2011min} for trigonometric polynomials in one dimension: three polynomials of maximal degree $2$ on $[0,1]$ (in green), with their maximum (in blue), and the upper-bounding polynomial (in red),  
 when using polynomials of degree $r$, with $r=2$ (left), $r=4$ (middle), and $r=8$ (right).
 \label{fig:lasserre}}
 \end{figure}

\subsection{Primal-dual SOS relaxation}
We consider the following ``exact'' reformulation already presented in \eq{F1}:
\[
\min_{x \in \X} \max_{y \in \Y} \ g(x,y) 
=
\min_{ \mu \in \mathcal{P}(\X)} \max_{ \nu: \X \to \mathcal{P}(\Y)}
\int_\X \int_\Y g(x,y) d\nu(y|x)  d\mu(x),
\]
where the maximization problem is replaced by the maximization of an expectation.
 Using the expression of $g$ in \eq{A}, we can then use the bi-linearity of $g$ and write the equation above as
  \[
  \min_{ \mu \in \mathcal{P}(\X)} \max_{ \nu: \X \to \mathcal{P}(\Y)}
\int_\X  \tr \bigg[ G \bigg( \int_\Y \psi(y)\psi(y)^\top d\nu(y|x)  \otimes \varphi(x) \varphi(x)^\top \bigg) \bigg]d\mu(x)
,  \]
  and thus
  as (with no approximation yet), with $\ds V(x) =  \int_\Y \psi(y)\psi(y)^\top d\nu(y|x)   \in \K_\psi$ (the hull of all $\psi(y)\psi(y)^\top$ for $y \in \Y$):
  \BEQ
\label{eq:B}  
\min_{x \in \X} \max_{y \in \Y} \ g(x,y) 
= \min_{ \mu \in \mathcal{P}(\X)} \max_{ V: \X \to \textcolor{red}{\K_\psi}}
\int_\X  \tr \big[ G \big(  V(x) \otimes \varphi(x) \varphi(x)^\top  \big) \big] d\mu(x).
  \EEQ
  We will now make a sequence of three relaxations to approximate the problem in \eq{B} above.
  
  \paragraph{Replacing $ {\mathcal{K}_\psi}$ by $\widehat{\mathcal{K}}_\psi$.}
  We first consider functions $V$ with values in $\widehat{\K}_\psi$ (which is computationally more manageable) instead of~${\K}_\psi$ (which may not), leading to
  \BEQ
\label{eq:C}  \min_{ \mu \in \mathcal{P}(\X)} \max_{ V: \X \to  \textcolor{red}{\widehat{\K}_\psi}}
\int_\X  \tr \big[ G \big( V(x)  \otimes \varphi(x) \varphi(x)^\top \big) \big] d\mu(x),
  \EEQ
  which is always greater or equal to the optimal value of \eq{B}.
  
  \paragraph{Parameterizing  $V$ by a matrix sum-of-squares.}
  The set $\widehat{\K}_\psi = \{ T \succcurlyeq 0, \ \tr(T) = 1, T \in \mathcal{V}_\psi\}$ has a PSD constraint. 
   Thus, as presented in \mysec{matrixSOS}, following~\cite{scherer2006matrix,fang2021sum,muzellec2021learning}, we can try to approximate it by a positive linear form in $\varphi(x) \varphi(x)^\top$ as
\[
V(x)  =T [  \varphi(x) \varphi(x)^\top ],
  \]
with $ T \in \rb^{m p \times mp }$ such that $T \succcurlyeq 0$, where, for $M \in \rb^{m\times m}$,  $T [M]$ denotes the symmetric matrix in $\rb^{ p \times p}$ such that for any symmetric matrix $N \in \rb^{p \times p}$, $\tr \big( 
T [M] N \big) = \tr \big[ T ( N \otimes M) \big]$. In other words, if $T$ is defined by blocks $T_{ij} \in \rb^{m \times m}$ for $i,j \in \{1,\dots,p\}$, then $T [  \varphi(x) \varphi(x)^\top ]_{ij} = \varphi(x)^\top T_{ij} \varphi(x)$.

 In order to impose that for all $x \in \X$,  $V(x) \in \V_\psi$ and $\tr [ V(x) ] = 1$, we add the additional affine constraints
 \[
  T \in  \V_\varphi^\perp \otimes \mathbb{S}_p + \mathbb{S}_m \otimes \V_\psi ,   \  \  \widetilde{\tr} [T]  - \idm \in \V_\varphi^\perp,
\]
where $\widetilde{\tr} [T] \in \rb^{m \times m} $ denotes the ``partial trace'' defined as $\tr( M \widetilde{\tr} [T]  ) = \tr( T( M \otimes \idm))$ for any matrix $M \in \rb^{ m \times m}$, and $\mathbb{S}_p$ denotes the set of symmetric matrices of size $p$.  In other words, $(\widetilde{\tr} [T]) _{ij} = \tr( T_{ij})$.

  We then obtain a problem where the measure $\mu$ only appears through the moment $\Sigma$ of $\varphi(x)^{\otimes 4} = \varphi(x)\varphi(x)^\top\otimes \varphi(x)\varphi(x)^\top \in \rb^{m^2 \times m^2}$, since we have, using properties of Kronecker products:
    \BEAS
  \int_\X  \tr \big[ G \big( V(x)  \otimes \varphi(x) \varphi(x)^\top \big) \big] d\mu(x)
  & = & \sum_{i,j=1}^p \int_\X \varphi(x)^\top G_{ij} \varphi(x) \varphi(x)^\top T_{ij} \varphi(x) d\mu(x) 
  \\
  & = & \tr \Big( \Sigma \sum_{i,j=1}^p  G_{ij} \otimes T_{ij} \Big).
  \EEAS

  We thus get a partially relaxed formulation, which cannot be solved yet by a semi-definite program (SDP), because of the set $\mathcal{K}_{\varphi\otimes \varphi}$:
   \BEA
\label{eq:partialSOS}
\min_{\Sigma \in \rb^{m^2 \times m^2}} \max_{T \in \rb^{mp \times mp}}
 \tr \Big( \Sigma \sum_{i,j=1}^p  \! G_{ij} \otimes T_{ij} \Big) &  \mbox{ such that} &  \ \Sigma  \in \mathcal{K}_{\varphi\otimes \varphi}\\[-.2cm]
\notag & &\  T \succcurlyeq 0, \ T \in  \V_\varphi^\perp \otimes \mathbb{S}_p + \mathbb{S}_m\otimes \V_\psi ,  \ \widetilde{\tr} [T] - \idm \in \V_\varphi^\perp.
\EEA
\paragraph{Replacing $ {\mathcal{K}_{\varphi \otimes \varphi} }$ by $\widehat{\mathcal{K}}_{\varphi \otimes \varphi}$.}
We obtain our final formulation, which can be solved as an SDP, where $ \mathcal{K}_{\varphi\otimes \varphi}$
 is replaced by  $\widehat{\mathcal{K}}_{\varphi\otimes \varphi}$:
   \BEA
   \label{eq:fullSOS}
\min_{\Sigma \in \rb^{m^2 \times m^2}} \max_{T \in \rb^{mp \times mp}}
 \tr \Big( \Sigma \sum_{i,j=1}^p  \! G_{ij} \otimes T_{ij} \Big) &  \mbox{ such that} &  \ \Sigma \succcurlyeq 0 , \  \Sigma \in \V_{\varphi \otimes \varphi},  \ \tr(\Sigma) = 1 \\[-.2cm]
\notag & &\  T \succcurlyeq 0, \ T \in  \V_\varphi^\perp \otimes \mathbb{S}_p + \mathbb{S}_m\otimes \V_\psi ,  \ \widetilde{\tr} [T] - \idm \in \V_\varphi^\perp.
\EEA
We thus obtain a convex-concave min-max problem corresponding exactly to \eq{minmaxfull} in \mysec{duminmax}.

\paragraph{Alternative formulations.}
We can then choose to transform it into a minimization problem akin to \eq{min} by adding a Lagrange multiplier $C$ for the constraint $T \succcurlyeq 0$, and $\Lambda \in \V_\varphi$ for $\widetilde{\tr} [T] - \idm \in \V_\varphi^\perp$, leading to:
\BEA
\label{eq:minSOSpb}
\min_{\Sigma \in \rb^{m^2 \times m^2} , \ C \in \rb^{mp \times mp}, \ \Lambda \in \rb^{m \times m} } \tr [ \Lambda ] 
& \mbox{ such that } &   \Sigma \Hsquare G  +  C - \Lambda \otimes I \in {\V_\varphi} \otimes \V_\psi^\perp \\
\notag & &  \Sigma \in \V_{\varphi \otimes \varphi}, \ \Sigma \succcurlyeq 0, \ \tr(\Sigma) = 1 \\
\notag & & \Lambda \in \mathcal{V}_\varphi, \ C \succcurlyeq 0,
\EEA
where $\Sigma \Hsquare G \in \rb^{mp \times mp}$ is defined by block as:
$[\Sigma \Hsquare G ]_{ij} = \Sigma_{ij} G_{ij} \in \rb^{m \times m}$. This is the formulation used for solving the optimization problem empirically in \mysec{experiments}.

 Alternatively, we obtain a maximization problem akin to \eq{max} from \eq{fullSOS}, by adding a Lagrange multiplier $A \succcurlyeq 0$ for the constraint $\Sigma \succcurlyeq 0$, and $c \in \rb$ for the constraint $\tr \Sigma = 1$:
 \BEA
\label{eq:maxSOSpb}
\max_{T \in \rb^{mp \times mp},\ A \in \rb^{m^2 \times p^2}, \ c \in \rb } c
& \mbox{ such that } & T \circ G  - c \idm - A \in  \V_{\varphi \otimes \varphi}^\perp \\
\notag & & T \succcurlyeq 0, \ T \in  \V_\varphi^\perp \otimes \mathbb{S}_p+ \mathbb{S}_m \otimes \V_\psi ,  \ \widetilde{\tr} [T] - \idm \in \V_\varphi^\top \\
\notag & & A \succcurlyeq 0.
\EEA
This formulation will be used in the convergence proof in \mysec{guarantees}.

\paragraph{Summary.}
Overall, the formulation is obtained through 3 approximations: 
\BIT
\item Replacing $ {\mathcal{K}_\psi}$ by $\widehat{\mathcal{K}}_\psi$. If this approximation is exact, then the maximization in $y$ is exact, and we obtain lower bounds. This is, for example the case for $\Y = \{1,\dots,p\}$, and also for degree $2$ polynomials. Otherwise, the equal values of problems in \eq{fullSOS}, \eq{minSOSpb}, and \eq{maxSOSpb} may be above or below the optimal value.

\item Parameterizing all functions $V: \X \to \mathcal{V}_\psi \cap \mathbb{S}_p^+$ by a matrix sum-of-squares~\cite{scherer2006matrix,fang2021sum,muzellec2021learning}. This can only be exact if the function $V$ is a polynomial, with the extra approximation due to the potential non-tightness of matrix SOS.

\item Replacing $ {\mathcal{K}_{\varphi \otimes \varphi} }$ by $\widehat{\mathcal{K}}_{\varphi \otimes \varphi}$. This is a typical SOS relaxation problem.
\EIT

These approximations are discussed in \mysec{guarantees}.

\subsection{Kernelization}
We assume that we have $m$ points $x_1,\dots,x_m$ such that the corresponding kernel matrix is invertible, complemented by $m'-m$ points $x_{m+1},\dots,x_{m'}$ such that $\V_\varphi$ is spanned by $\varphi(x_1) \varphi(x_1)^\top, \dots, 
\varphi(x_{m'}) \varphi(x_{m'})^\top$, and finally $m''-m'$ points such that
$\V_{\varphi \otimes \varphi}$ is spanned by $\varphi(x_1)^{\otimes 4}, \dots, 
 \varphi(x_{m''})^{\otimes 4}$. We denote by $K' \in \rb^{m' \times m'}$ the kernel matrix of the first $m'$ points. The matrix $K'$ is not invertible, but $K' \circ K'$ (with $\circ$ the element-wise product) is, because 
 $\varphi(x_1) \varphi(x_1)^\top, \dots, 
\varphi(x_{m'}) \varphi(x_{m'})^\top$ is a basis of $\V$. We denote by $K'' \in \rb^{m' \times m''}$ the kernel matrix between the $m'$ first points and all $m''$ points. 
 
 We can then express for all $i \in \{1,\dots,m''\}$,
  $\varphi(x_i) \varphi(x_i)^\top = \sum_{j=1}^{m'} N_{ji} \varphi(x_j) \varphi(x_j)^\top$, with the matrix $N
  \in \rb^{ m' 
  \times m''}$ equal to
  $N =( K' \circ K' ) ^{-1} K''$.
  
We can write \eq{minSOSpb} with $\Sigma = \sum_{i=1}^{m''} \alpha_i  \varphi(x_{i})^{\otimes 4} $, $C = \sum_{i=1}^{m'} D_i \otimes\varphi(x_i)\varphi(x_i)^\top$   , and $\Lambda 
= \sum_{i=1}^{m'} \lambda_i\varphi(x_i)\varphi(x_i)^\top $, and get the optimization problem (which is an SDP which we use in our experiments \mysec{experiments}): 
\BEAS
 & & \min_{\alpha \in \rb^{m''} , \ \lambda \in \rb^{m'}, \ D_1,\dots,D_{m'} \in \rb^{p' \times p'} } \sum_{i=1}^{m'} \lambda_i \\
 &  \mbox{ such that } & 
\forall i \in \{1,\dots,m'\}, j \in \{1,\dots,p'\},  \ 
\psi(y_j)^\top D_i \psi(y_j) - \lambda_i + [ N \Diag(\alpha) G ]_{ij} = 0 \\
& &  \sum_{i=1}^{m'}   D_i  \otimes \varphi(x_i)\varphi(x_i)^\top\succcurlyeq 0, \ 
 \sum_{i=1}^{m''} \alpha_i = 1, \ 
\sum_{i=1}^{m''} \alpha_i  \varphi(x_{i}) ^{ \otimes 4} \succcurlyeq 0.
\EEAS
From the vector $\alpha$, we can obtain a potential minimizer using algorithms from \mysec{practical}, with the possibility of full kernelization as in \mysec{kernelization}, where $G \in \rb^{m'' \times p'}$ is the matrix of evaluations $g(x_i,y_j)$.

\subsection{A posteriori guarantees}
Since we have used relaxations of both maximization and minimization problems, we do not obtain, in general, an upper or lower bound, except in some special cases that we now describe.

If the feature map $\psi$ is such that $\widehat{\K}_\psi  = \K_\psi$, then from the matrix $T \in \rb^{mp \times mp}$, we get $V: \X \to \K_\psi$, and thus a feasible dual point for 
\eq{B}. Therefore the value of the SOS formulation is always below the true one.
If $\Sigma$ is represented by a singleton $\varphi(x_\ast) \varphi(x_\ast)^\top$, then if $V(x_\ast)$ is such that 
$\max_{y \in \Y} L(x_\ast,y) = \tr[ M ( \varphi(x_\ast)\varphi(x_\ast)^\top \otimes V(x_\ast) )]$, we have a tight solution. This happens in our simulations.

If  $\widehat{\K}_\psi  \supsetneq \K_\psi$, then, when $\Sigma$ is represented by a singleton $\varphi(x_\ast) \varphi(x_\ast)^\top$, if $V(x_\ast)$ is such that 
$\max_{y \in \Y} L(x_\ast,y) = \tr[ M ( \varphi(x_\ast)\varphi(x_\ast)^\top \otimes V(x_\ast) )]$, we only know that we have an upper-bound on the true value.

\subsection{A priori guarantees}

\label{sec:guarantees}

In this section, we focus primarily on the situation where $\widehat{\K}_\psi = \K_\psi$, so we do not have to use hierarchies on $\Y$. If this is not the case, we can use another relaxation $\widetilde{K}_\psi$ that would lead to an extra approximation factor that goes to zero as the degree of the hierarchy on $y$ goes to infinity. Still, the precise details are out of the scope of this paper.

The main new result shows that for the polynomial examples in \mysec{examples}, the partially relaxed problem in \eq{partialSOS} can be solved through hierarchies with arbitrary precisions. We make the following assumptions:
\BIT
\item[\textbf{(A1)}] $\widehat{\K}_\psi = \K_\psi$, so that our relaxation is a lower-bound.
\item[\textbf{(A2)}] Given a one-dimensional Lipschitz-continuous function $g: \X \to \rb$, it can be approximated by a quadratic form in 
$  { \varphi \choose \varphi_1^+  } $, where $\varphi_1^+$ includes additional monomials, and we denote by $\varepsilon^{{\rm app}}(\varphi,\varphi_1^+)$ the approximation constant so that for all $g$, there exists a quadratic form in $\tilde{\varphi} $ defined by the matrix $\tilde{H}$, such that
\[
\forall x \in \X, \ | g(x) - \tilde{\varphi}(x)^\top \tilde{H}  \tilde{\varphi}(x) | \leqslant {\rm Lip}(g) \cdot \varepsilon^{{\rm app}}(\varphi,\varphi^+_1).
\]
It is known that, given $\varphi$, we can make $\varepsilon^{{\rm app}}(\varphi,\varphi^+_1)$ as small as desired by increasing the degree of the polynomials, with well-studied convergence rates, through ``Jackson's inequalities''~\cite{ganzburg1981multidimensional}.

\item[\textbf{(A3)}] We solve the equivalent optimization problems in \eq{fullSOS}, \eq{minSOSpb}, or \eq{maxSOSpb} with $\varphi$ replaced by $\tilde{\varphi} = { \varphi \choose \varphi^+  } $, where $\varphi^+ = { \varphi_1^+ \choose \varphi_2^+}$ includes additional mononials on top of $\varphi_1^+$. We denote by $\varepsilon^{\rm SOS}(\varphi,\varphi^+_1,\varphi_2^+)$ the SOS approximability ratio defined in \mysec{sosguarantees} as, for any $H \in \rb^{(m+m_1)s \times (m+m_1)s}$:
\[
\!\! \forall x \in \X, \ H[ \varphi(x) \varphi(x)^\top] \succcurlyeq  \varepsilon^{\rm SOS}(\varphi,\varphi^+_1,\varphi_2^+)  \| H \|_{\rm F}\idm \ \ \Rightarrow \ \ \exists \tilde{A} \succcurlyeq 0, \ \forall x \in \X, \  H[ \varphi(x) \varphi(x)^\top] = 
 \tilde{A}[ \tilde{\varphi}(x) \tilde{\varphi}(x)^\top].
\]
We select the threshold to have a similar result for scalar-valued quadratic forms in $\varphi(x)\varphi(x)^\top$, as described at the end of \mysec{matrixSOS}.
We know from \mysec{matrixSOS} that, given $\varphi$ and $\varphi^+_1$, we can make $\varepsilon^{\rm SOS}(\varphi,\varphi^+_1,\varphi_2^+)$ as small as desired by increasing the degree of the polynomials.  \EIT 

Note that we only need to divide $\varphi^+$ in ${\varphi^+_1 \choose \varphi^+_2}$ for the proof, as the algorithm is oblivious to this distinction. Our main result follows.
\begin{theorem}
Let $G$ be defined by a polynomial as in \eq{A} and $\varepsilon > 0$. Assume \textbf{(A1)},   \textbf{(A2)}, and  \textbf{(A3)}. Then there exist feature maps $\varphi^+_1$ and $\varphi^+_2$ such that the optimal value of the SOS primal-dual relaxation is within $\varepsilon$ of the optimal value.
\end{theorem}
\begin{proof}
This requires obtaining SOS polynomial approximations to the matrix-valued function $V$ obtained in \eq{C}. To obtain a finite convergence, we would need to represent one of the many optimal~$V$'s exactly. Here we will consider a specific approximation based on Von Neumann entropy regularization and start with a smoothing lemma.
\begin{lemma}
\label{lemma:smoothing}
Let $B \in \mathbb{S}_p$ and $\eta>0$. Let $W_\eta(B)$ be the unique maximizer of $\tr[ BW ] - \eta \tr [ W \log W  ]$ such that $W \succcurlyeq 0$, $ \tr(W)=1$, and $W \in \mathcal{V}_\psi \subset \rb^p $. Then $W_\eta$ is a $(1/\eta)$-Lipschitz-continuous function of $B$, and
\[
0 \leqslant  \ \max_{W \succcurlyeq 0, \ \tr(W) = 1, \ W \in \mathcal{V}_\psi}
\tr[ BW ]  - \tr[ BW_\eta(B)]  \ \leqslant \eta \log p.
\]
\end{lemma}
\begin{proof}
Since the function $W\mapsto \tr [ W \log W]$ is 1-strongly convex with respect to the nuclear norm on the set $\{ W \succcurlyeq 0, \tr(W) = 1\}$~\cite{yu2013strong}, $W_\eta$ is such that $\| W_\eta(B) - W_\eta(B') \|_\ast \leqslant \frac{1}{\eta} \| B - B'\|_{\rm op}$~\cite{lemarechal1997practical}, where $\| \cdot \|_\ast$ denotes the nuclear norm. The bound is obtained by looking at eigenvalues of $V$ and using classical bound on entropies~\cite{cover1999elements}.
\end{proof}
We can now build a feasible point for \eq{maxSOSpb} which will lead to the desired bound. Following the discussion at the end of \mysec{duality}, there are many optimal candidates for $x \mapsto V(x)$. In this proof, we propose a dual candidate $V$ based on maximizing approximately $g(x,y)$ for \emph{all} $x \in \X$, and not only at the minimizer $x_\ast$. While it allows to show asymptotic convergence, it is not sufficient to show finite convergence.

We consider $\eta>0$, and the function $V: x \mapsto W_\eta \big( G[ \varphi(x) \varphi(x)^\top ] )$ obtained from Lemma~\ref{lemma:smoothing}. By construction, $\forall x \in \X, V(x) \in \widehat{\K}_\psi$, $V$ is Lipschitz-continuous with constant proportional to $\| G\|_{\rm F} /\eta$ (with constants that depends on $\varphi$), 
since
\[
\| V(x') - V(x)\|_\ast \leqslant \frac{1}{\eta}
\| G'[ \varphi(x) \varphi(x')^\top ] - G[ \varphi(x) \varphi(x)^\top ] \|_{\rm op}
\leqslant \frac{1}{\eta} \| G\|_{\rm F} {\rm Lip}(\varphi).
\]
Moreover, from Lemma~\ref{lemma:smoothing}, we have
\BEQ
\label{eq:appratio}
 0 \leqslant \max_{y \in \Y} g(x,y) - \tr \big( V(x) G[ \varphi(x) \varphi(x)^\top ] \big) \leqslant \eta \log p.
 \EEQ

We can then use  Assumption \textbf{(A2)} and approximate $V$ by a quadratic form in $ {\varphi \choose \varphi^+_1}$. We thus find a  matrix $U \in \rb^{(m+m_1)p \times (m+m_1)p}$ such that all affine constraints on values of $V$ are still satisfied, that is, 
$ U \in  \V_\varphi^\perp \otimes \mathbb{S}_p+ \mathbb{S}_m \otimes \V_\psi  $ and $\widetilde{\tr} [U] - \idm \in \V_\varphi^\top $, and
\[
\forall x \in \X, \ \| V(x) - U[ \varphi(x) \varphi(x)^\top ] \|_{\rm op} \leqslant  C \cdot
 \frac{1}{\eta} \| G\|_{\rm F}  \cdot  \varepsilon^{({\rm app})}(\varphi,{\varphi}^+_1)
\]
for some constant $C$ that is independent of $G$ and $\varphi^+_1$. Thus for all $x \in \X$, $U[ \varphi(x) \varphi(x)^\top ]\succcurlyeq 
- C \cdot
 \frac{1}{\eta} \| G\|_{\rm F}  \cdot  \varepsilon^{{\rm app}}(\varphi,{\varphi}^+_1) \idm $, which implies from Assumption \textbf{(A3)} that 
\[
U[ \varphi(x) \varphi(x)^\top] + \big[ \varepsilon^{{\rm SOS}}(\varphi,\varphi^+_1,\varphi_2^+)  \| U\|_{\rm F}
+  C \cdot
 \frac{1}{\eta} \| G\|_{\rm F}  \cdot  \varepsilon^{{\rm app}}(\varphi,{\varphi}^+_1)\big] \idm
\]
is a sum-of-squares and satisfies all other affine constraints. We denote by $T$ the corresponding matrix, which is feasible for \eq{maxSOSpb}. 
Moreover, because of \eq{appratio}, we have, for all $x \in \X$,
\[
\tr \big( T[ \varphi(x) \varphi(x)^\top ]G[ \varphi(x) \varphi(x)^\top ] \big) 
\geqslant  \min_{x' \in \X} \max_{y \in \Y} g(x',y) - \eta \log p - \big[ \varepsilon^{{\rm SOS}}(\varphi,\varphi^+_1,\varphi_2^+)  \| U\|_{\rm F}
+   
 \frac{C}{\eta} \| G\|_{\rm F}  \cdot  \varepsilon^{{\rm app}}(\varphi,{\varphi}^+_1)\big],
\]
and thus, applying Assumption~\textbf{(A3)} again, $T \circ G - \min_{x' \in \X} \max_{y \in \Y} g(x',y) -  \eta \log p 
+
\big[  \varepsilon^{{\rm SOS}}(\varphi,\varphi^+_1,\varphi_2^+) ( \| U\|_{\rm F} + \| T \circ G\|_{\rm F}) 
+  
 \frac{C}{\eta} \| G\|_{\rm F}  \cdot  \varepsilon^{{\rm app}}(\varphi,{\varphi}^+_1)\big]$ is a sum of squares, and thus we obtain an approximation of 
$\min_{x' \in \X} \max_{y \in \Y} g(x',y) $ up to 
$\eta \log p +  \varepsilon^{{\rm SOS}}(\varphi,\varphi^+_1,\varphi_2^+)   ( \| U\|_{\rm F} + \| T \circ G\|_{\rm F}) 
+   
 \frac{C}{\eta} \| G\|_{\rm F}  \cdot  \varepsilon^{{\rm app}}(\varphi,{\varphi}^+_1)$.
Now, given $\varepsilon >0$, we take $\eta = \frac{1}{3 \log p}$, then select $\varphi^+_1$ such that 
$ 
 \frac{C}{\eta} \| G\|_{\rm F}  \cdot  \varepsilon^{{\rm app}}(\varphi,{\varphi}^+_1)$ is smaller than $\varepsilon/3$, and finally select $\varphi_2^+$ such that $ \varepsilon^{{\rm SOS}}(\varphi,\varphi^+_1,\varphi_2^+)  ( \| U\|_{\rm F} + \| T \circ G\|_{\rm F}) 
$ is less than $\varepsilon/3$. This leads to desired approximation within $\varepsilon$.
\end{proof}

We can make the following observations:
\BIT
\item The hierarchy often empirically converges in finitely many iterations, but we cannot find provable sufficient conditions. It would be interesting to see if, assuming that the polynomial is convex-concave, we could use tools from \cite{de2011lasserre} to prove such convergence. 
\item To obtain a convergence rate, we would need to be able to characterize the dependence of $\varepsilon^{{\rm SOS}}(\varphi,\varphi^+_1,\varphi_2^+)$ on $\varphi_1^+$, which we leave for future work.
\EIT
 
\subsection{Special case $\Y = \{1,\dots,p\}$}
This corresponds to having $\mathcal{K}_\psi = \widehat{\K}_\psi$ the set of PSD diagonal matrices with unit trace. We can then simplify notations and solve
\[
\ds
\min_{x \in \X} \max_{j \in \{1,\dots,p\}} \varphi(x)^\top G_j \varphi(x),
\]
with $G_1,\dots,G_p \in \mathbb{S}_m$.
We then have $V$ diagonal, with diagonal elements $v_j(x) = \varphi(x)^\top T_j \varphi(x)$, with $T_j \succcurlyeq 0$, and $  \sum_{j=1}^ p T_j - \idm \in \V_\varphi^\perp$.
We thus obtain the min/max formulation corresponding to \eq{fullSOS}:
 \BEA
\label{eq:fullfull}\min_{\Sigma \in \rb^{m^2 \times m^2}} \max_{T_1,\dots,T_p \in  \in \rb^{m \times m}}
\tr \Big[ \Sigma \cdot \sum_{j=1}^p G_j \otimes T_j \Big] &  \mbox{ such that} &  \ \Sigma \succcurlyeq 0 , \  \Sigma \in \V_{\varphi \otimes \varphi},  \ \tr(\Sigma) = 1 \\[-.2cm]
\notag & &\  T_1,\dots,T_p \succcurlyeq 0, \  \sum_{j=1}^ p T_j - \idm \in \V_\varphi^\perp.
\EEA
We also get the minimization formulation, which is the one used in experiments, corresponding to \eq{minSOSpb}:
\BEA
\label{eq:minmin} \min_{\Sigma \in \rb^{m^2 \times m^2} , \  \Lambda \in \rb^{m \times m} } \tr [ \Lambda ] 
& \mbox{ such that } &  \forall j \in \{1,\dots,p\}, \ \Sigma [ G_j ] \preccurlyeq \Lambda \\
\notag & &  \Sigma \in \V_{\varphi \otimes \varphi}, \ \Sigma \succcurlyeq 0, \ \tr(\Sigma) = 1 \\
\notag & & \Lambda \in \mathcal{V}_\varphi.
\EEA
We also get a maximization formulation, corresponding to \eq{maxSOSpb}, and leading to a nice interpretation below:
 \BEA
 \label{eq:maxmax}
\max_{T_1,\dots,T_p \in    \rb^{m\times m}}
c  &  \mbox{ such that} &
\sum_{j=1}^p G_j \otimes T_j  - c \idm - A \in  \V_{\varphi \otimes \varphi}^\perp \\[-.2cm]
\notag & &\  T_1,\dots,T_p \succcurlyeq 0, \  \sum_{j=1}^ pT_j - \idm \in \V_\varphi^\perp.
\EEA

\paragraph{Kernelization.} Empirically, we solve
\BEAS
\min_{\alpha \in \rb^{m''} , \ \lambda \in \rb^{m'} } \sum_{i=1}^{m'} \lambda_i
& \mbox{ such that } &  \forall j \in \{1,\dots,p\}, \ \sum_{i=1}^{m''} \alpha_i g_j(x_i) \varphi(x_i) \varphi(x_i)^\top \preccurlyeq \sum_{i=1}^{m'} \lambda_i \varphi(x_i) \varphi(x_i)^\top  \\[-.2cm]
& &   \sum_{i=1}^{m''} \alpha_i = 1, \ \sum_{i=1}^{m''} \alpha_i  \varphi(x_{i})\varphi(x_{i})^\top \otimes 
\varphi(x_{i})\varphi(x_{i})^\top \succcurlyeq 0.
\EEAS
We obtain the matrices $T_1,\dots,T_p$ as Lagrange multipliers for the PSD constraints.

\paragraph{Relationship with Putinar's Positivstellensatz.} An interesting parallel with Putinar's Positivstellensatz~\cite{putinar1993positive} can be made. We consider $p$ multi-variate polynomials $g_1,\dots,g_p$, with $\X \subset \rb^d$   one of the simple sets described in \mysec{sos}. Because of the approximation result in \mysec{guarantees}, we know that if $\min_{x \in \X} \max_{j=1,\dots,p} g_j(x)$ is strictly positive, there is a level of the hierarchy of polynomials so that our relaxation also has strictly positive values, and, in fact, the converse is also true.
Thus, using the maximization formulation from \eq{maxmax}, $\min_{x \in \X} \max_{j=1,\dots,p} g_j(x) >0$,
if and only if there exists $c>0$ and sum-of-square (that is, PSD quadratic forms in $\varphi$) polynomials $q_0$ (represented by $A$), and $q_1,\dots,q_p$, represented by $T_1,\dots,T_p$, such that
\[
\forall x \in \X, \ 
c = \sum_{j=1}^p g_j(x) q_j(x) - q_0(x) ,
\]
and such that $q_1(x) + \cdots + q_p(x) = 1$ for all $x \in \X$.

Without loss of the generality, we can take $c=1$, and we have shown that 
\[
\min_{x \in \X} \max_{j=1,\dots,p} g_j(x) >0
\]
 if and only if there exist SOS polynomials (based on the feature vector $\varphi$) $q_0,\dots,q_p$ such that 
\[
\forall x \in \X, \ 
- 1 = \sum_{j=1}^p (-g_j(x)) q_j(x) + q_0(x)  
\]
and $q_1(x) + \cdots + q_p(x) $ is constant on $\X$.

Without the last constraint, this turns out  to be exactly the Putinar certificate for the positivity of $-1$ on the set 
$\mathcal{A} = \{x \in \rb^d, \ \forall j \in \{1,\dots,p\}, \ -g_j(x) \geqslant 0\}$, and thus a certificate for the emptiness of that set. Given that 
\[
\min_{x \in \X} \max_{j=1,\dots,p} g_j(x) \leqslant 0 
\ \Leftrightarrow\  \exists x \in \X, \ \forall j \in \{1,\dots,p\},\  g_j(x)\leqslant 0 \ \Leftrightarrow \
\X \cap \mathcal{A} \neq \varnothing,
\]
we obtained a feasibility certificate similar to the one obtained for Putinar. Note that the original Putinar certificate does require an extra assumption, e.g., that one of the sets 
$\{x \in \rb^d,\  - g_j(x) \geqslant 0\}$ is bounded.

Note moreover that with our assumptions from \mysec{examples}, SOS-polynomials that are PSD quadratic forms in $\varphi$ have a slightly different meaning; that is, for example, for the unit Euclidean ball, they correspond to $(1-\|x\|_2^2) u(x) + v(x)$, where $u$ and $v$ are regular sums-of-squares. This leads to the following proposition (where we have replaced $g_j$ by $-g_j$ to match classical certificates, and we have dropped the constraint of summing to a constant, which is not necessary).

\begin{proposition}
Let $g_1,\dots,g_p$ be $p$ multivariate polynomials on $\rb^d$. Then, the set
\[
\{x \in \rb^d, \ \|x\|_2^2 \leqslant 1, \ \forall j \in \{1,\dots,p\}, \ g_j(x) \geqslant 0\}
\]
is empty if and only if there exists polynomials $u_0,v_0, u_1,v_1,\dots,u_p,v_p$ that are sums-of-squares such that
\[
\forall x \in \rb^d, \
- 1 = \sum_{j=1}^p g_j(x) \big[ (1-\|x\|_2^2) u_j(x) + v_j(x) \big]  + (1-\|x\|_2^2) u_0(x) + v_0(x)  .
\]
\end{proposition}
It is weaker than Putinar's certificate, which would not need $v_0$, and $u_1,\dots,u_p$. Still, it could be extended to continuous situations where the set $\Y$ (and the corresponding feature map $\psi$) in our min-max formulation leads to tight SOS formulations, for example, for polynomials in $[-1,1]$.

 \section{Experiments}
\label{sec:experiments}

In this section, we provide illustrative experiments where we obtain tight relaxations on small problems. See~{\small \url{https://www.di.ens.fr/~fbach/sos_min_max.zip}} for Matlab code reproducing these experiments.

\begin{figure}
\begin{center}
\includegraphics[width=14cm]{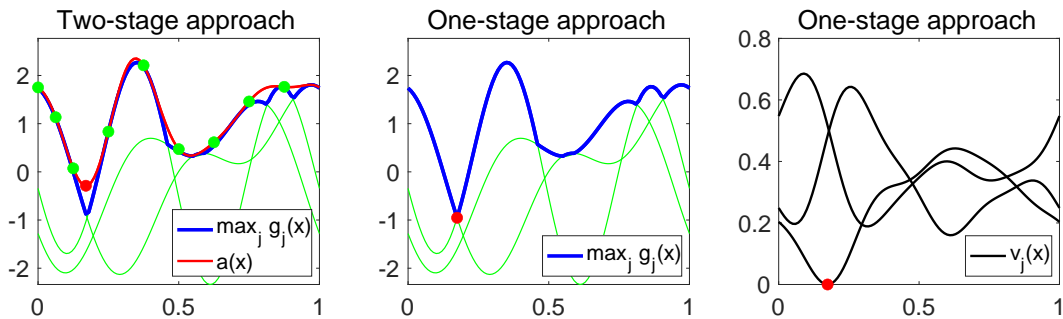}
\end{center}

\vspace*{-.3cm}

\caption{Minimization of the maximum of 3 trigonometric polynomials on $[0,1]$: two-stage approach of~\cite{lasserre2011min} (left), one-stage primal-dual approach (middle), functions $v_j$, $j=1,2,3$ from the two-stage approach (right).
\label{fig:1dmax}}

\end{figure}

\begin{figure}
\begin{center}
\includegraphics[width=10.2cm]{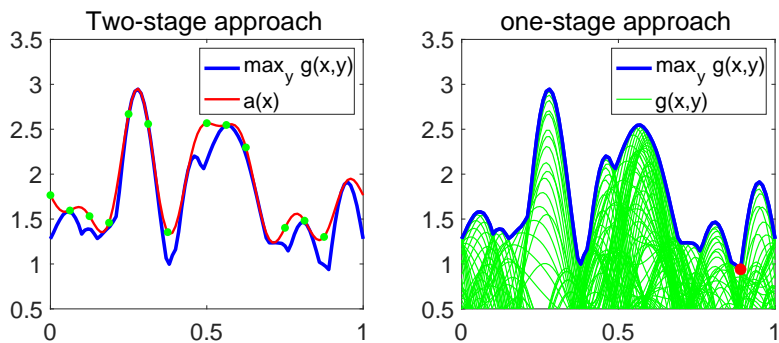}
\end{center}

\vspace*{-.3cm}

\caption{Minimization of the maximum of a bivariate trigonometric polynomial with $\X = \Y = [0,1]$. Two-stage approach of~\cite{lasserre2011min} (left), one-stage primal-dual approach (right).
\label{fig:1dpol}}

\end{figure}

\begin{figure}
\begin{center}
\includegraphics[width=10cm]{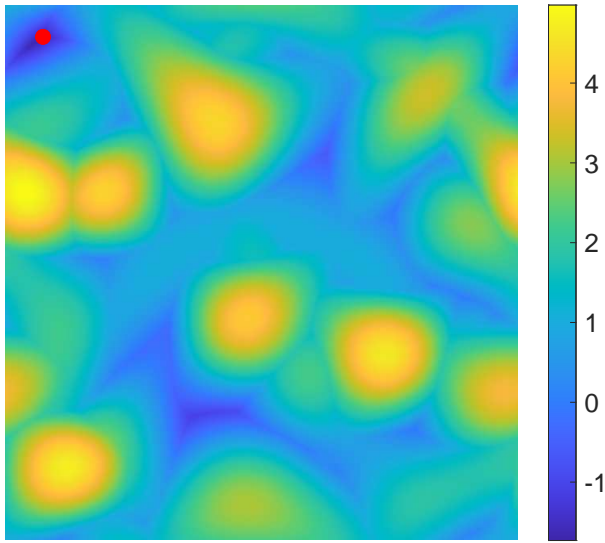}
\end{center}

\vspace*{-1cm}

\caption{Minimization of the maximum 4 trigonometric polynomials on $[0,1]^2$, with the maximizer in red.
\label{fig:2dmax}}

\end{figure}

\paragraph{Minimizing the maximum of univariate trigonometric polynomials.} See~\myfig{1dmax}, where we obtain a tight relaxation. We also plot the optimal function $v_j(x) = \varphi(x)^\top T_j \varphi(x)$, $j \in \{1,\dots,p\}$, which are non-negative and sum to one, and, at $x_\ast$, have non-zero values only for the $j$'s attaining the maximum in $\max_{j \in \{1,\dots,p\}} g_j(x_\ast)$.

\paragraph{Maximizing the maximum of bivariate trigonometric polynomials.} See~\myfig{2dmax} for an example with a tight relaxation.

 \paragraph{Min-max optimization of a trigonometric polynomial on $[0,1]^2$.} See~\myfig{1dpol} for an example with a tight relaxation.

 \section{Conclusion}
 In this paper, we proposed an SOS formulation for min-max problems over polynomials and provided a convergence proof when degrees of polynomials are allowed to increase. This work opens up several avenues for future work, such as (a) 
infinite-dimensional extensions for smooth functions by adding proper regularization like done by~\cite{rudi2020finding} for plain minimization, (b) finding sufficient conditions for either finite convergence or an explicit rate, and (c) exploring how the min-max approach relates to the several Positivstellensatz from the literature.

 \subsection*{Acknowledgements}
The author thanks Didier Henrion for helpful discussions about Putinar feasibility certificates and Jean-Bernard Lasserre for insightful comments on an earlier version of this paper. The author acknowledges support
from the French government under the management of the Agence Nationale de la {Recher-che} as part of the ``Investissements d’avenir'' program, reference ANR-19-P3IA0001 (PRAIRIE 3IA Institute). This work was also supported by the European Research Council (grant SEQUOIA 724063).

\appendix

\section{Convergence rates of matrix-valued SOS}
\label{app:matrixSOS} 
We extend the proof of~\cite[Theorem 1]{bach2022exponential} to matrix-valued polynomials, using the same technique as~\cite{fang2021sum}, and following the notations of \cite{bach2022exponential} closely.

\begin{proposition}
\label{theo:nocond}
Let $r>0$ and $s \geqslant 3r$, and $\varepsilon(s) = \big[  \big( 1 - \frac{6r^2}{s^2} \big)^{-d} - 1 \big] \sim_{s \to +\infty}    \frac{ 6 r^2 d }{s^2}$.
For any multivariate matrix-valued trigonometric polynomial $f$ of degree less than $2r$, written
$f(x) = \sum_{\| \omega \|_\infty \leqslant 2r} \hat{f}(\omega) e^{2i\pi \omega^\top x}$,
\[
\forall x \in [0,1]^d, f(x) \succcurlyeq \varepsilon(s) \sum_{\| \omega \|_\infty \leqslant 2r, \ \omega \neq 0}\!\! \| \hat{f}(\omega)  \|_{\rm op} \ \Rightarrow \  f  \ \mbox{is a sum of squares of polynomials of degree } s.
\]
\end{proposition}
\begin{proof}
 We consider the following integral operator on $1$-periodic  matrix-valued functions on $ [0,1]^d$, defined as 
 \BEQ
 \label{eq:Th}
 Th(x) = \int_{[0, 1]^d} |q(x-y)|^2 h(y) dy,
 \EEQ
 for a well-chosen $1$-periodic function $ q$ which is a trigonometric polynomial of degree~$ s$. 
 {The function $x \mapsto |q(x-y)|^2$ is an element of the finite-dimensional cone of SOS polynomials of degree $s$, thus,} by design, if $ h$ has positive semi-definite values, then $ Th$ is a sum of squares of matrix polynomials of degree less than~$ s$. We will find $ h$ such that $ Th = f  $.
 
 In the Fourier domain, since convolutions lead to pointwise multiplication and vice-versa, we have for all $\omega \in \zb^d$, where $\hat{q} \ast \hat{q}(\omega)$ is a shorthand for $(\hat{q} \ast \hat{q})(\omega)$ : 
\[
\widehat{Th}(\omega) = \hat{q} \ast \hat{q}(\omega) \cdot \hat{h}(\omega),
\]
 and thus, the candidate $ h$ is defined by its Fourier series, which is equal to zero for $ \| \omega\|_\infty > 2r$, and to 
 \[ \frac{\hat{f}(\omega) }{ \hat{q}  \ast \hat{q}(\omega)}\]
  otherwise. If we impose that $ \hat{q}  \ast \hat{q}(0)=1$, we then have 
\BEAS
   f  -  h &  = &    \sum_{\omega  \in \zb^d} \hat{f}(\omega) \Big(  1  - \frac{1}{\hat{q} \ast \hat{q}(\omega)} \Big) \exp( 2i \pi \omega^\top \cdot) =    \sum_{\omega  \neq 0 } \hat{f}(\omega) \Big(  1  - \frac{1}{\hat{q} \ast \hat{q}(\omega)} \Big) \exp( 2i \pi \omega^\top \cdot) .\EEAS
 We then get:
\BEQ
\label{eq:Ap} \sup_{x \in [0,1]^d} \| f(x)  - h(x)\|_{\rm op}
 \leqslant    \sum_{\omega  \neq 0} 
\big\| \hat{f}(\omega)  \big\|_{\rm op}
\cdot  \max_{ \|\omega\|_\infty \leqslant 2r}  \Big|   \frac{1}{\hat{q} \ast \hat{q}(\omega)}  \, - 1\Big| .\EEQ

With the choice 
$\ds  \hat{q}(\omega) = a \prod_{i=1}^d \Big( 1 - \frac{|\omega_i|}{s} \Big)_+, $
with $a$ a normalizing constant, we get $ \hat{q}\ast \hat{q}(0)=1$ and
$ \max_{ \|\omega\|_\infty \leqslant 2r}  \big|   \frac{1}{\hat{q} \ast \hat{q}(\omega)}  \, - 1\big| \leqslant \varepsilon(s)$ (see~\cite{bach2022exponential} for details). Thus, for all $x \in [0,1]^d$, using \eq{Ap} and the assumption on $f$:
\[
h(x) = f(x) - ( f(x) - h(x)) \succcurlyeq 
 \varepsilon(s) \sum_{ \omega \neq 0} \| \hat{f}(\omega)  \|_{\rm op}
 -  \varepsilon(s) \sum_{ \omega \neq 0} \| \hat{f}(\omega)  \|_{\rm op} = 0,
\]
which leads to the desired result.
 \end{proof}

 \section{Alternating optimization for the two-stage approach}
 \label{app:alternate}
 In this section, we explore briefly the possibility evoked in \mysec{sosex} of trying to minimize \eq{jb} with respect to $\Sigma$ as well. This is a non-convex problem, and alternating optimization has a particularly simple formulation.
  Indeed, in the kernelized version in \eq{sosexk}, this corresponds to replacing $\mu$ by the previous value of $\alpha$ and iterating. Since the first upper-bound is minimized exactly, at the second iteration and all later ones, the matrix $\Sigma$ corresponds to a Dirac measure, and the upper-bounding polynomial is so that its value at this point is minimized.  This is shown empirically in \myfig{lasserre_alt}: even in the good attraction basin, the alternating optimization does not lead to the global optimum.

   \begin{figure}
\begin{center}
 \includegraphics[width=12cm]{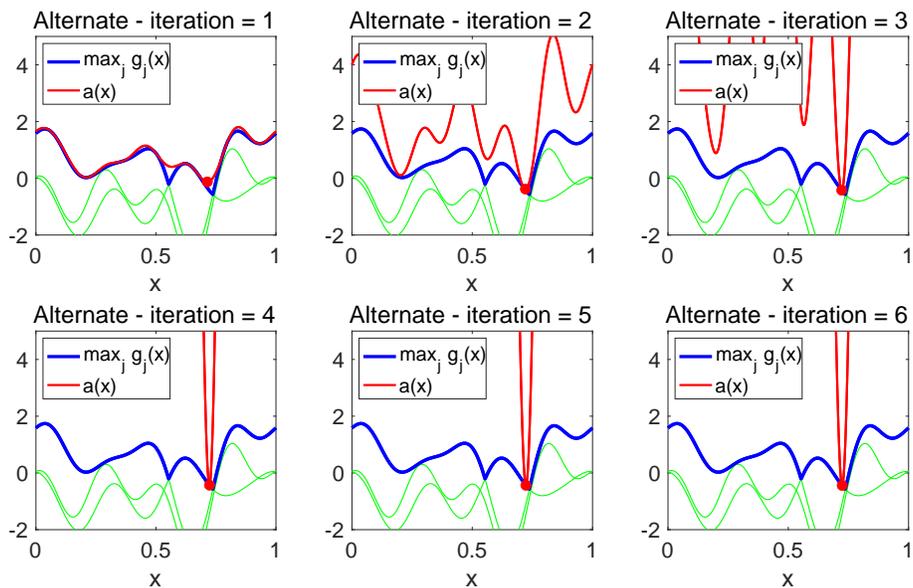}
\end{center}
\vspace*{-.45cm}

 \caption{Two-stage approach for trigonometric polynomials in one dimension, with alternating optimization and $r=2$, with  6 iterations.
 \label{fig:lasserre_alt}}
 \end{figure}

  \bibliography{SOS_min_max}

 \end{document}